\documentclass{article}
\usepackage{ucs}
\usepackage{amsthm, amsmath, amssymb}
\usepackage[T1]{fontenc}
\usepackage[utf8]{inputenc}
\usepackage{amsfonts}
\usepackage[english]{babel}

\newtheorem*{conjecture}{Conjecture} 
\newtheorem*{id}{Identity}
\newtheorem{theorem}{Theorem} 
\newtheorem{lemma}{Lemma}
\newtheorem{cor}{Corollary}
\theoremstyle{remark}
\newtheorem{rem}{Remark}[theorem]

\begin{document}
\newcommand{\legendre}[2]{\left(\frac{#1}{#2}\right)}
\newcommand{\Gal}{\operatorname{Gal}}
\newcommand\myeq{\stackrel{\mathclap{\normalfont\mbox{def}}}{=}}

\author{Dmitry Krachun}
\title{On sums of triangular numbers}
\maketitle

Let us denote the $k$-th triangular number by $T_k$, that is, $T_k=\frac{k(k+1)}{2}$.
The goal of this note is to prove the following results:

\begin{theorem}
Any non-negative integer number $n$ can be written in the form $$a(2a-1)+b(2b-1)+c(2c+1)+d(2d+1),$$ where $a,b,c,d\in \mathbb{N}_0$.
\label{theorem1}
\end{theorem}

\begin{theorem}
Any non-negative integer number $n$ can be written in the form $$2a(2a-1)+b(2b-1)+2c(2c+1)+d(2d+1),$$ where $a,b,c,d\in \mathbb{N}_0$.
\label{theorem2}
\end{theorem}

\begin{rem}
Since $a(2a-1)=T_{2a-1}$ and $c(2c+1)=T_{2c}$, the first theorem is essentially about representing $n$ in the form $T_x+T_y+T_z+T_t$ where two of the numbers $x,y,z,t$ are even and two are odd, whereas the second theorem is about representing numbers in the form $2T_x+2T_y+T_z+T_t$ where $x-y$ and $z-t$ are odd.
\end{rem}

\begin{rem}
The facts have been conjectured by \emph{Zhi-Wei Sun}, see [2, Conj.5.3].
\end{rem}
In order to prove the first theorem, we are going to use one lemma.
\begin{lemma}
Any number which is not of the form $4^l(8k+7)$ can be written as $a^2+b^2+c^2$ with non-negative $a,b,c$.
\label{s+s+s}
\end{lemma}

\begin{proof}
This result is due to Legendre. See, for example, [1, Th. 5.4.12].
\end{proof}

\begin{cor}
Any number $m$ can be written as $x^2+T_y+T_z.$
\label{s+t+t}
\end{cor}

\begin{proof}
Since $4m+1$ is not of the form $4^l(8k+7)$, we have $4m+1=a^2+b^2+c^2$. Two of the numbers $a,b,c$ have different parity, suppose they are $b$ and $c$, then $a$ is even and we obtain 
\begin{align*}
m&=\frac{(8m+2)-2}{8}=\frac{2a^2+(b+c)^2-1+(b-c)^2-1}{8}\\
&=(a/2)^2+T_{\frac{b+c-1}{2}}+T_{\frac{b-c-1}{2}}.
\end{align*}
\end{proof}

\begin{cor}
Any number $m$ can be written as $4T_x+T_y+T_z.$
\end{cor}

\begin{proof}
Here we write $8m+6=a^2+b^2+c^2$, by Lemma \ref{s+s+s} as $8m+6$ is not of the form $4^l(8k+7)$. Considering it modulo $8$, we can assume that $b,c$ are odd and 
$a \equiv 2 \pmod4$. We obtain $$m=\frac{(8m+6)-6}{8}=\frac{4((a/2)^2-1)+ b^2-1+c^2-1}{8}=4T_{\frac{a-2}{4}}+T_{\frac{b-1}{2}}+T_{\frac{c-1}{2}}.$$
\end{proof}

\begin{cor}
$2T_x+T_y+T_z$ represents all non-negative integers.
\label{2t+t+t}
\end{cor}
\begin{proof}
We can write $4m+2=a^2+b^2+c^2$, two of the numbers $a,b,c$ have different parity, suppose they are $b$ and $c$, then $a$ is odd and we obtain 
\begin{align*}
m&=\frac{(8m+4)-4}{8}=\frac{2(a^2-1)+(b+c)^2-1+(b-c)^2-1}{8}=\\
&=2T_a+T_{\frac{b+c-1}{2}}+T_{\frac{b-c-1}{2}}.
\end{align*}
\end{proof}

\begin{id}
$a^2+2T_b=T_{a+b}+T_{a-b-1}=T_{a+b}+T_{b-a}.$
\label{iden}
\end{id}
\begin{proof}
The proof is just a straightforward computation.
\end{proof}
\begin{rem}
We shall use the expression $T_{a+b}+T_{a-b-1}$ if $a>b$ and the expression $T_{a+b}+T_{b-a}$ if $a\leq b$. 
\end{rem}

\begin{proof}[Proof of Theorem \ref{theorem1}]
We shall prove the fact for all $n>200$ as it is easy to check the statement for all numbers up to $200$. 

Suppose first that the number $n$ is odd.
Consider the biggest number $c$ such that $2c^2+2c+1\leq n$. Then $n-2c^2-2c-1$ is even and using Corollary \ref{2t+t+t} we obtain 
$$\frac{n-2c^2-2c-1}{2}=2T_d+T_x+T_y,$$ which can be rewritten as 
\begin{align*}
n&=2c^2+2c+1+2d^2+2d+2T_x+2T_y=\\
&=(c+d+1)^2+(c-d)^2+2T_x+2T_y.
\end{align*}

Now, since $c$ is the biggest number such that $2c^2+2c+1\leq n$, we have $2(c+1)^2+2(c+1)+1>n$, and thus $2T_d+T_x+T_y=\frac{n-2c^2-2c-1}{2}<\frac{2(c+1)^2+2(c+1)+1-2c^2-2c-1}{2}=2c+2$. Since $n>200$, we have $c\geq 9$ and therefore $$d+x<2\sqrt{T_x+T_d}<2\sqrt{2c+2}\leq c,$$ which enables us to conclude that $c-d>x$ and $c+d+1>y$. We obtain
\begin{align*}
n&=(c+d+1)^2+2T_y+(c-d)^2+2T_x=\\
&=T_{c+d+1+y}+T_{c+d-y}+T_{c-d+x}+T_{c-d-x-1}.
\end{align*}

The case when $n$ is even is almost the same.
Consider the biggest number $c$ such that $2c^2\leq n$. Then using Corollary \ref{s+t+t} we obtain 
$$n-2c^2=2(p^2+T_x+T_y),$$ which can be rewritten as 
$$n=(c-p)^2+(c+p)^2+2T_x+2T_y.$$
Now, since $c$ is the biggest number such that $2c^2\leq n$, we have $2(c+1)^2>n$, and thus $p^2+T_x+T_y=\frac{n-2c^2}{2}<\frac{2(c+1)^2-2c^2}{2}=2c+1$. Since $n>200$, we have $c\geq 10$ and therefore $$p+x<2\sqrt{T_x+p^2}<2\sqrt{2c+1}\leq c,$$ which enables us to conclude that $c-p>x$ and $c+p>y$. We obtain 
\begin{align*}
n&=(c-p)^2+2T_x+(c+p)^2+2T_y=\\
&=T_{c-p+x}+T_{c-p-x-1}+T_{c+p+y}+T_{c+p-y-1}.
\end{align*}

\end{proof}
In order to prove the second theorem, we are going to use some more technical lemmas.
\begin{lemma}
If $n$ is a sum of two odd squares and $n$ is divisible by $5^2, 13^2$ or $61^2$, then $n$ can be written as  $a^2+b^2$ with $a$ and $b$ being odd and having different remainders modulo 4.
\label{differentparity}
\end{lemma}
\begin{proof}
We shall treat three cases separately. \newline
First, suppose $n$ is divisible by $25$. Let $n/25=p^2+q^2$ with $p\geq q$, then $n=(5p)^2+(5q)^2=(4p-3q)^2+(3p+4q)^2$. We have $5p-5q \equiv p-q \pmod4$, whereas $(4p-3q)-(3p+4q) \equiv p+q \pmod4$. But either $p-q$ or $p+q$ is not divisible by $4$.\newline
Second, suppose $n$ is divisibly by $169$. Let $n/169=p^2+q^2$ with $p\geq q$, then $n=(13p)^2+(13q)^2=(12p-5q)^2+(5p+12q)^2$. We have $13p-13q \equiv p-q \pmod4$, whereas $(12p-5q)-(5p+12q) \equiv 3(p+q) \pmod4$. But either $p-q$ or $p+q$ is not divisible by $4$.\newline
Finally, suppose $n$ is divisibly by $61^2$. Let $n/61^2=p^2+q^2$ with $p\geq q$, then $n=(61p)^2+(61q)^2=(60p-11q)^2+(11p+60q)^2$. We have $61p-61q \equiv p-q \pmod4$, whereas $(60p-11q)-(11p+60q) \equiv p+q \pmod4$. But either $p-q$ or $p+q$ is not divisible by $4$.

\end{proof}

\begin{lemma}
 If $8n+3$ is divisible by $5^2, 13^2$ or $61^2$, then $n$ can be written in the form $T_x+T_y+T_z$, where $x$ and $y$ have different parity.
 \label{t+t+t}
\end{lemma}
\begin{proof}
Suppose that $t^2|8n+3$, $t \in \{5,13,61\}$. Using Lemma \ref{s+s+s} we can write $\frac{8n+3}{t^2}=p^2+q^2+r^2$ with odd $p,q$ and $r$. Using Lemma \ref{differentparity} we can write $t^2(p^2+q^2)$ as 
a sum of two odd squares $a^2+b^2$ with $a,b$ having different remainders modulo $4$. Then we obtain 
$$n=\frac{(8n+3)-3}{8}=\frac{a^2-1+b^2-1+(tr)^2-1}{8}=T_{\frac{a-1}{2}}+T_{\frac{b-1}{2}}+T_{\frac{tr-1}{2}}.$$
Clearly, $\frac{a-1}{2}$ and $\frac{b-1}{2}$ have different parity.
\end{proof}
\begin{lemma}
If $8n+6$ is divisible by $5^2, 13^2$ or $61^2$, then $n$ can be written in the form $T_x+T_y+4T_z$, where $x$ and $y$ have different parity.
\label{t+t+4t}
\end{lemma}
\begin{proof}
Suppose that $t^2|8n+6$, $t \in \{5,13,61\}$. Using Lemma \ref{s+s+s} we can write $\frac{8n+6}{t^2}=p^2+q^2+r^2$. Considering this modulo $8$ we understand that $r$ is congruent to $2$ modulo $4$, and $p,q$ are odd. Using Lemma \ref{differentparity} we can write $t^2(p^2+q^2)$ as a sum of two squares $a^2+b^2$ with $a,b$ having different remainders modulo $4$. Then we obtain 
\begin{align*}
n&=\frac{(8n+6)-6}{8}=\frac{a^2-1+b^2-1+4((tr/2)^2-1)}{8}=\\&=T_{\frac{a-1}{2}}+T_{\frac{b-1}{2}}+4T_{\frac{tr-2}{4}}.
\end{align*}
Most obviously, $\frac{a-1}{2}$ and $\frac{b-1}{2}$ have different parity.
\end{proof}

\begin{lemma}
If $m$ is of the form $p^2+q^2$ with $p,q$ being odd and having different remainders modulo $4$, then so is $3965\cdot m$.
\label{oddsquares}
\end{lemma}
\begin{proof}
Indeed, suppose that $q\leq p$, then $3965\cdot m = (53^2+34^2)(p^2+q^2)=(53p-34q)^2+(34p+53q)^2$. Both numbers are odd and $(53p-34q)+(34p+53q) \equiv 3p+3q \pmod4.$
\end{proof}

\begin{lemma}
If $m$ is of the form $p^2+q^2$, where $p$ is even, $q$ is odd, and $p>q$ , then $3965\cdot m$ can be written in the same form.
\label{evensquare+oddsquare}
\end{lemma}
\begin{proof}
Suppose first that $p>5q$, then $3965\cdot m = (59^2+22^2)(p^2+q^2)=(59p-22q)^2+(22p+59q)^2$ and $59p-22q>22p+59q$; otherwise we have
$3965\cdot m = (46^2+43^2)(p^2+q^2)=(46p-43q)^2+(43p+46q)^2$ and $43p+46q>46p-43q$. 
\end{proof}

We now prove the second theorem.

\begin{proof}[Proof of Theorem \ref{theorem2}]
For $n<4\cdot10^8$ the statement can be more or less easily checked straightforwardly.\newline
Suppose first that there exists $t\in \{5,13,61\}$ such that $4n+3$ is not divisible by $t$.
Let us first assume that $4n+3$ is a square modulo $t$, which means that $4n+3$ is a square modulo $t^2$, thus there exists $A$ such that $(4n+3)-4A^2=8\cdot \frac{n-A^2}{2}+3$ is divisible by $t^2$. Moreover, we can choose $A$ such that $\sqrt{n}\geq A \geq \sqrt{n}-t^2$ and $n-A^2$ is even. Then using Lemma \ref{t+t+t} we can write $$\frac{n-A^2}{2}=T_x+T_y+T_z,$$ where $x$ and $y$ have different parity. We now rewrite this as $$n=A^2+2T_z+2T_x+2T_y=T_{A+z}+T_{A-z-1}+2T_x+2T_y,$$ where $A+z, A-z-1$ have different parity as well as $x,y$ do, which basically is the desired form. We only need to show that $A\geq z$ or $A^2\geq 2T_z$. For that it is sufficient to show that $A^2\geq n-A^2$ or $\sqrt{n}-t^2\geq \sqrt{n/2}$, which is true for all $n$ greater than $(6+4\sqrt{2})t^4\leq 2\cdot10^8$.

If $4n+3$ is not a square modulo $t$ then it is a doubled square modulo $t$, as $\left(\frac{2}{t}\right)=-1$. Thus $4n+3$ is a doubled square modulo $t^2$ and there exists $A$ such that $4n+3-8A^2$ is divisible by $t^2$. Moreover, we can choose $A$ such that $\sqrt{n/2}\geq A \geq \sqrt{n/2}-t^2$. Then $8(n-2A^2)+6$ is divisible by $t^2$ and using Lemma \ref{t+t+4t} we can write $$n-2A^2=T_x+T_y+4T_z,$$ where $x$ and $y$ have different parity. We now rewrite this as $$n=2(A^2+2T_z)+T_x+T_y=2T_{A+z}+2T_{A-z-1}+T_x+T_y,$$ where $A+z$, $A-z-1$ have different parity as well as $x,y$ do, which is what we need. We are left with showing that $A\geq z$ or $2A^2\geq 4T_z$, for that it is sufficient to show that $2A^2\geq n-2A^2$ or $\sqrt{n/2}-t^2\geq \sqrt{n/4}$, which is true for all $n$ greater than $(12+8\sqrt{2})t^4\leq 4\cdot 10^8$.

We now treat the case when $4n+3$ is divisible by $3965=5 \cdot 13 \cdot 61$. The original problem is equivalent to representing $n$ in the form $2T_x+2T_y+T_z+T_t=2T_x+2T_y+(\frac{z+t+1}{2})^2+2T_{\frac{z-t-1}{2}}$ or $4n+3=(2x+1)^2+(2y+1)^2+4A^2+(2s+1)^2$ where $x-y$ is odd and $A>s$. We may assume that the fact is already proved for all smaller $n$, then writing $\frac{4n+3}{3965}=(2x_0+1)^2+(2y_0+1)^2+4A_0^2+(2s_0+1)^2$ and using Lemma \ref{oddsquares} and Lemma \ref{evensquare+oddsquare} we obtain the desired form for $4n+3$.
\end{proof}

Finally we formulate an ambitious conjecture.

\begin{conjecture}
Any natural number which is equal to neither $8$, nor $68$ can be written in either of the forms:
$$a(2a-1)+b(2b-1)+c(2c+1), \qquad a(2a-1)+b(2b+1)+c(2c+1),$$
where $a,b,c\in \mathbb{N}_0.$
\end{conjecture}

\textbf{Acknowledgment.} I would like to thank Alexander Luzgarev for fruitful discussions and helpful comments on the proof of Theorem \ref{theorem2}.

\end{document}